\documentclass{article}
\usepackage[british]{babel}
\usepackage{fontenc}[T1]
\usepackage{amssymb}
\usepackage{array}
\usepackage{latexsym}
\usepackage{enumerate}
\usepackage{amsmath}
\usepackage{amsfonts}
\usepackage{amsthm}
\date{}
\newcommand{\NN}{\mathbb N}
\newcommand{\ZZ}{\mathbb Z}

\newcommand{\fB}{\mathfrak B}
\newcommand{\fC}{\mathfrak C}
\newcommand{\fW}{\mathfrak W}
\newcommand{\fS}{\mathfrak S}
\newcommand{\Aut}{\operatorname{Aut}}
\newcommand{\lcm}{\operatorname{lcm}}
\newcommand{\Stab}{\operatorname{Stab}}
\theoremstyle{plain}
\newtheorem{prop}{Proposition}[section]
\newtheorem{theorem}[prop]{Theorem}
\newtheorem{corollary}[prop]{Corollary}
\newtheorem{lemma}[prop]{Lemma}
\theoremstyle{definition}
\newtheorem{remark}[prop]{Remark}
\newtheorem{example}[prop]{Example}
\newtheorem{definition}[prop]{Definition}
\author{L. Giuzzi$\,^{*}$ \and A. Pasotti\thanks{
    {\tt luca.giuzzi@ing.unibs.it},
    {\tt anita.pasotti@ing.unibs.it},
    Dipartimento di Matematica,
    Facolt\`a di Ingegneria,
    Universit\`a degli Studi di Brescia,
    Via Valotti 9,
    I-25133 Brescia (IT).
    This
    research has been partially supported by ``Fondazione Tovini'',
    Brescia.}
}
\title{Sampling complete designs}
\begin{document}
\maketitle
\begin{abstract}
Suppose $\Gamma'$ to be of a subgraph of a graph $\Gamma$.
We define a \emph{sampling} of a $\Gamma$-design ${\fB}=(V,B)$ into a
$\Gamma'$-design ${\fB'}=(V,B')$ as a surjective map $\xi:B\to B'$
mapping any block of $B$ into one of its subgraphs.
A sampling will be called {\it regular} when the number of
preimages of any block of $B'$ under $\xi$ is a constant.
This new
concept is closely related with the classical notion of
\emph{embedding}, which  has been extensively studied,
for many classes of graphs,
by several authors; see, for example, the survey \cite{Q2002a}.
Actually, a sampling $\xi$ might induce several embeddings
of the design $\fB'$ into $\fB$, although the converse is not
true in general.
In the present paper we study in more detail the behaviour of
samplings of $\Gamma$--complete designs of order $n$ into
$\Gamma'$--complete designs of the same order and show how
the natural necessary condition for the existence of a
regular sampling is actually sufficient. We also provide
some explicit constructions of samplings, as well as
propose further generalizations.
\end{abstract}


\section{Introduction}
Denote by $K_n$ the complete graph on $n$ vertices, and assume
$\Gamma\leq K_n$ to be a subgraph of $K_n$.
Write $\,^{\lambda}K_n$ for the multigraph obtained from $K_n$ by
repeating each of its edges exactly $\lambda$ times.
 A
\emph{$(\,^{\lambda}K_n,\Gamma)$--design}  is a set $\fB$ of
graphs, called \emph{blocks}, isomorphic to $\Gamma$ and
partitioning the edges of $\,^{\lambda}K_n$. Recall that an
automorphism  of $\fB$ is a permutation of the vertices of $K_n$
leaving $\fB$ invariant. These designs, especially those with
$\lambda=1$ and endowed with a rich automorphism group, are a
topic of current research; several constructions, as well as
existence results, are known; see, for instance, \cite{BO,BZ,AP}.
We generalize the classical concept of
 $(v,k)$--complete design, see \cite{Cam,MR}, to
the context of graph decompositions.
\begin{definition}
\label{deD}
  Suppose $\Gamma\leq K_n$.
  By a \emph{$(K_n,\Gamma)$--complete design} we mean
  the set $K_n(\Gamma)$ consisting of all subgraphs of
  $K_n$ isomorphic to $\Gamma$.
\end{definition}
\noindent Clearly, a $(K_v,K_k)$--complete design is just a
$(v,k)$--complete design.
\par
The main topic of this paper is the new concept of \emph{sampling}
of designs; we propose the following definition.
\begin{definition}
Given a $(\,^{\lambda}K_n,\Gamma)$--design $\fB$ and a
$(\,^{\mu}K_v,\Gamma')$--design
$\fB'$ with $\Gamma'\leq\Gamma$, a \emph{$\fB'$--sampling} of
$\fB$ is a surjective function $\xi:\fB\to\fB'$ such that
$\xi(B)\leq B$, for any $B\in\fB$.
\end{definition}
\noindent  When $n=v$ and $\lambda=\mu=1$, we shall speak simply
of $\Gamma'$--samplings of $\fB$. In general, to prove the
existence of samplings of arbitrary designs is an interesting yet
difficult problem. In the present paper,
 we shall be concerned with
 samplings of complete designs.
\par
Our notion of sampling is closely related to that of embedding of
designs. We recall that an \emph{embedding} of $\fB'$ into $\fB$,
see \cite{Q2002a}, is a function $\psi:\fB'\to\fB$ such that for
any $G\in\fB'$,
\[ G\leq \psi(G). \]
In recent years, embeddings have been extensively investigated and
several results have been obtained for various classes of graphs;
see for instance, \cite{CLQ,GQ,LQR,LCQ,MQ95,MQ99,M,Q2002,Q2003}.
An injective embedding is called \emph{strict}. Note that given a
sampling $\xi:K_n(\Gamma)\to K_n(\Gamma')$, there is always a
strict embedding $\psi:K_n(\Gamma')\to K_n(\Gamma)$ such that
$\xi\psi$ is the identity on $K_n(\Gamma')$. However, the reverse,
namely that any strict embedding induces a sampling, is not true,
unless the embedding is taken to be bijective.

Samplings are also related to nestings of cycle systems;
see Section \ref{nc} for some details.


We also propose this new definition for samplings between complete
designs.
\begin{definition}
A \emph{regular $\Gamma'$--sampling} of the complete design
$K_n(\Gamma)$
is a sampling $\xi:K_n(\Gamma)\to K_n(\Gamma')$
such that the number of preimages of any
$G\in K_n(\Gamma')$ is a constant $\lambda>0$.
Such a sampling is said to have \emph{redundancy $\lambda$.}
\end{definition}
Observe that constructing a regular $\Gamma'$--sampling of
redundancy $\lambda$ is the same as to extract from every block of
$K_n(\Gamma)$ a subgraph isomorphic to $\Gamma'$ in such a way as
to cover the set $K_n(\Gamma')$ exactly $\lambda$ times.

We propose an analogous definition for embeddings.
\begin{definition}
  A \emph{$\lambda$--fold regular embedding} of $K_n(\Gamma')$ in $K_n(\Gamma)$
  is an embedding $\psi:K_n(\Gamma')\to K_n(\Gamma)$ such that
  any $G\in K_n(\Gamma)$ has  $\lambda$ preimages
  under $\psi$.
\end{definition}
\noindent The main result of the present paper is contained in
Theorem \ref{mainTG}: it shows that the natural necessary
conditions for the existence of regular $\Gamma'$--samplings and
$\lambda$--fold regular embeddings for $K_n(\Gamma)$ are also
sufficient. Theorem \ref{Ref} shows that when the hypotheses of
Theorem \ref{mainTG} are not fulfilled, there might still exist,
in some cases, samplings with some regularity.

In Section \ref{s4},  we shall
provide some direct constructions of regular samplings, using
suitable automorphism groups, and focusing our attention to the
case in which both $\Gamma$ and $\Gamma'$ are complete graphs. We
will also purpose a generalisation of the notion of
$(K_n,\Gamma)$--complete design to arbitrary graphs.

\section{Samplings and nestings}
\label{nc}
In this section we use standard notations of graph theory;
see \cite{H}.
In particular, by $C_m$ we mean the cycle of length $m$,
whereas $S_{m+1}$ denotes the star with $m$ rays and $m+1$
vertices; with $W_{m+1}$ we write the wheel with $m+1$
vertices, that is the graph obtained from a cycle $C_m$,
by adding a further vertex adjacent to all the preexisting
ones.

An \emph{$m$--cycle system $\fC$ of order $n$} is just a
$(K_n,C_m)$--design; see \cite{BR}. A \emph{nesting} of $\fC$ is a
function $f:\fC\to V(K_n)$ such that
\[ \fS=\left\{\{x,f(C)\}\,|\, C\in \fC, x\in V(C) \right\} \]
is a $(K_n,S_{m+1})$--design.

Nestings of cycle systems have been extensively studied; see
\cite{GL,GMR,LRS89,LRS,LS,RS,S}.
Observe that, by construction, a nesting of $\fC$ always
determines a bijection
$g:\fC\to\fS$.
We may  consider the
set
\[ \fW=\{ C\cup g(C)\,|\, C\in \fC \}. \]
All elements of $\fW$ are wheels $W_{m+1}$. It is immediate to see
that $\fW$ is a $(\,^2K_n,W_{m+1})$--design. Since both $C_m\leq
W_{m+1}$ and $S_{m+1}\leq W_{m+1}$, there exists at least two
samplings, $\xi_1:\fW\to\fC$ and $\xi_2:\fW\to\fS$. Actually, it
is possible to reconstruct the nesting from just $\xi_2$, as shown
in the following proposition.
\begin{prop}
\label{pp}
  Suppose $\fW$ to be a $(\,^2K_n,W_{m+1})$--design and let
  $\fS$ be a $(K_n,S_{m+1})$--design.
  There exists a sampling $\xi:\fW\to\fS$ if, and only if,
  there is a nesting of an
  $m$--cycle system $\fC$ of order $n$.
\end{prop}
\begin{proof}
  For any $W\in\fW$, let $\zeta(W)=W\setminus\xi(W)$.
  Clearly $\zeta(W)$ is always an $m$--cycle.
  Since $\fW$
  is a $(\,^2K_n,W_{m+1})$--design and
  $\fS$ is a $(K_n,S_{m+1})$--design, the set
  \[ \fC=\{ \zeta(W)\,|\,W\in\fW \} \]
  is an $m$--cycle system of order $n$.
  Define now $f:\fC\to V(K_n)$ as the function
  which sends any cycle $C=\zeta(W)$ into
  the centre of the wheel $W$.
  The function $f$ is a nesting, as required.
\end{proof}
\begin{remark}
In the proof of Proposition \ref{pp}, it is essential
to have that $\zeta(W)=W\setminus\xi(W)$ is a cycle $C_m$.
In general, any wheel $W_{m+1}$ contains $m+1$ cycles $C_m$.
However, for $m>3$ only one of these cycles, say $C$, is such that
$W_{m+1}\setminus C$ is a star. Hence, a sampling $\xi':\fW\to\fC$,
is not, in general, associated with a sampling $\fW\to\fS$; thus,
we may not have a nesting.
\end{remark}

\section{Preliminaries on graph and matching  theory}
Here we recall some
known results about matchings of bipartite graphs;
for further references,
see \cite{LP,T}.

By graph we shall always mean a finite unordered graph
$\Gamma=(V,E)$ without loops, having vertex set $V$ and edge set
$E$.

Recall that a graph $\Gamma=(V,E)$ is \emph{bipartite} if $V$ can
be partitioned into two sets $\Gamma_1,\Gamma_2$ such that every
edge of $\Gamma$ has one vertex in $\Gamma_1$ and one in
$\Gamma_2$. The \emph{degree} of $v\in V$ in $\Gamma$ is the
number $\deg_{\Gamma}(v)$ of edges of $\Gamma$ containing $v$. A
bipartite graph $\Gamma$ with vertex set $\Gamma_1\cup\Gamma_2$ is
\emph{$(d,e)$--regular} if each vertex in $\Gamma_1$ has degree
$d$ while each vertex in $\Gamma_2$ has degree $e$. If $d=e$,  we
say that $\Gamma$ is \emph{regular of degree
  $d$}.
\par
A \emph{matching} in a graph $\Gamma$ is a set of edges of
$\Gamma$, no two of which are incident. A \emph{perfect matching}
(or \emph{$1$--factor})  of $\Gamma$ is a matching partitioning
the vertex set $V(\Gamma)$. In a bipartite graph $\Gamma$ with
vertex set $\Gamma_1\cup\Gamma_2$, a matching is \emph{full} if it
contains $\min(|\Gamma_1|,|\Gamma_2|)$ edges. Clearly, a full
matching of $\Gamma$ is perfect if, and only if,
$|\Gamma_1|=|\Gamma_2|$.

The following lemma shall be used throughout the paper.
\begin{lemma}[\cite{A}, pag. 397, Corollary 8.13]
\label{cor813}
Any bipartite $(d,e)$--regular graph $\Gamma$ possesses a full
matching.
\end{lemma}

An \emph{edge colouring} of $\Gamma$ is a function
$w:E\to\NN$ such that for any incident edges, say $e_1, e_2$,
\[ w(e_1)\neq w(e_2). \]
An \emph{$n$--edge colouring} of $\Gamma$ is an edge colouring
using exactly $n$ colours. The \emph{chromatic index}
$\chi'(\Gamma)$ is the minimum $n$ such that $\Gamma$ has an
$n$--edge colouring.

\begin{theorem}[K\"onig Line Colouring Theorem, \cite{K1,K2}]
\label{klct}
For any bipartite graph $\Gamma$,
\[ \chi'(\Gamma)=\max_{v\in \Gamma}\,\deg_{\Gamma}(v). \]
\end{theorem}
For a proof in English of this result and more references on the topics, see
\cite[Theorem 1.4.18]{LP} and the discussion therein.

\section{Embeddings and samplings of
  $(K_n,\Gamma)$--complete designs}
\label{s3} Throughout this section, let $\Gamma'\leq\Gamma$ be two
subgraphs of $K_n$.
\par
By $\,^{\lambda}K_n(\Gamma)$ we will denote the multiset obtained
from $K_n(\Gamma)$ by repeating each
  of its elements exactly $\lambda$ times.

\begin{lemma}
  \label{tteo}
  Let $b_1$ be the number of blocks of
  $K_n(\Gamma)$ and $b_2$ be that of
  $K_n(\Gamma')$ and let $m=\lcm(b_1,b_2)$. Then
  there is a bijective embedding
 \[ \psi:\,^{(m/b_2)}K_n(\Gamma') \to \,^{(m/b_1)}K_n(\Gamma). \]
\end{lemma}
\begin{proof}
  Introduce the bipartite graph $\Delta$ with vertex set
  $V=K_n(\Gamma)\cup K_n(\Gamma')$ and $x,y\in V$ are adjacent
  if, and only if, $x\neq y$ and either
  $x\leq y$ or $y\leq x$.

  In the first step of the proof, we verify that
  $\Delta$ is $(d,e)$--regular, for some $d,e\in\NN$.
  As the automorphism group of $K_n$ is $\Aut(K_n)\simeq S_n$,
  we have that $\Aut(K_n)$ is transitive on both
  $K_n(\Gamma)$ and $K_n(\Gamma')$.
  We now argue by way of contradiction.
  Suppose that there are $\Gamma_1,\Gamma_2\in K_n(\Gamma)$
  such that
  \[ d_1=\deg \Gamma_1<\deg \Gamma_2=d_2. \]
  Then, there exists $\sigma\in\Aut(K_n)$ such that
  $\sigma(\Gamma_2)=\Gamma_1$.
  In particular, the image under $\sigma$ of the $d_2$ subgraphs
  of $\Gamma_2$ isomorphic to $\Gamma'$ consists of $d_2$ subgraphs
  of $\Gamma_1$, all isomorphic to $\Gamma'$. However, we supposed
  the number of subgraphs of $\Gamma_1$ isomorphic to $\Gamma'$ to be
  $d_1<d_2$; this yields a contradiction.

  Likewise, suppose we have two graphs $\Gamma_1', \Gamma_2'\in K_n(\Gamma')$
  with
  \[ e_1=\deg \Gamma_1'<\deg \Gamma_2'=e_2. \]
  As $\Aut(K_n)$ is transitive on $K_n(\Gamma')$, there is
  $\sigma\in\Aut(K_n)$ with $\sigma(\Gamma_2')=\Gamma_1'$.
  This permutation $\sigma$, in particular, sends the
  $e_2$ graphs isomorphic to $\Gamma$ containing
  $\Gamma_2'$ into $e_2$ distinct graphs containing $\Gamma_1'$
  isomorphic to $\Gamma$. This yields a contradiction,
  since $e_1<e_2$.

  Let now $\Delta'$ be the graph obtained from $\Delta$ by
  replicating $(m/b_1)$--times  $K_n(\Gamma)$ and
  $(m/b_2)$--times $K_n(\Gamma')$.
  By construction, $\Delta'$ is a
  $(dm/b_2,em/b_1)$--regular bipartite graph.
  By Lemma \ref{cor813}, $\Delta'$ admits a full matching $M$.
  Furthermore, since both parts of $\Delta'$ have the same
  cardinality, $M$ is perfect.
  \par
  For any $x\in K_n(\Gamma')$, define
  $\psi(x)=y$ where $(x,y)\in M$.
  This provides an embedding, as required.
\end{proof}
\begin{remark}
\label{remE}
Since $\psi$ in Lemma \ref{tteo} is bijective, the function
\[  \xi=\psi^{-1}:\,^{(m/b_1)}K_n(\Gamma)\to \,^{(m/b_2)}K_n(\Gamma') \]
is a sampling.
\end{remark}

\noindent Now we are ready to prove our main result.
\begin{theorem}
  \label{mainTG}
  The  complete design $K_n(\Gamma)$
  admits a  regular $\Gamma'$--sampling if, and
  only if, there is $\lambda\in\NN$ such that
  \begin{equation}
    \label{rcond}
    |K_n(\Gamma)|=\lambda|K_n(\Gamma')|.
  \end{equation}
  The redundancy of any such sampling is $\lambda$.
  \par
  Conversely, there is a $\lambda$--fold
  regular embedding of $K_n(\Gamma')$
  in $K_n(\Gamma)$ if, and only if
  \[ |K_n(\Gamma')|=\lambda|K_n(\Gamma)|. \]
\end{theorem}
\begin{proof}
 Clearly, Condition \eqref{rcond} is necessary for the existence
 of a regular $\Gamma'$--sampling.
 By Remark \ref{remE}, there is a bijective sampling
 \[ \vartheta: K_n(\Gamma)\to\,^{\lambda}K_n(\Gamma'). \]
 Each $y\in K_n(\Gamma')$ appears exactly $\lambda$ times
 in $\,^{\lambda}K_n(\Gamma')$. As such, $y$ is the image
 of $\lambda$ elements of $K_n(\Gamma)$ under $\vartheta$.
 Thus, $\vartheta$ induces a regular $\Gamma'$--sampling
 $\xi:K_n(\Gamma)\to K_n(\Gamma')$ with redundancy $\lambda$.
 \par
 The second part of the theorem is proved in an analogous way,
 using the bijective embedding
 \[ \psi: K_n(\Gamma')\to\,^{\lambda}K_n(\Gamma) \]
 provided by Lemma \ref{tteo}.
\end{proof}
We have to point out that the proof of previous theorem is not
constructive. In order to determine actual samplings we may, in
general, need to use some of the algorithms for matchings in graphs,
like the ones in \cite{LP}, applied to the graph $\Delta'$ of
Lemma \ref{tteo}. Also, in Section \ref{s4}, we will construct
explicitly regular samplings for some complete designs.

When $|K_n(\Gamma')|$ is not a divisor of
$|K_n(\Gamma)|$,
the previous theorem
fails.
Under the assumption
\[ |K_n(\Gamma)|=\lambda |K_n(\Gamma')|+r, \]
with $0<r<|K_n(\Gamma')|$,
we may investigate
the existence of
sampling functions which are ``as regular as
possible'', namely in which the number of
preimages of any given element is either $\lambda$
or $\lambda+1$.
These samplings shall be called
\emph{ $(\lambda,\lambda+1)$--semiregular.}

We start with the following lemma.
\begin{lemma}
\label{lr0}
Suppose
\[ \lambda=\left\lfloor\frac{|K_n(\Gamma)|}{|K_n(\Gamma')|}\right\rfloor; \]
then,
there is a strict embedding
$\xi:\,^{\lambda}K_n(\Gamma')\to K_n(\Gamma)$.
\end{lemma}
\begin{proof}
  Argue as in the first part of the proof of Lemma \ref{tteo}
  and introduce the $(d,e)$--regular bipartite graph $\Delta$.
  Let now $\Delta'$ be the graph obtained from $\Delta$ by
  replicating
  $\lambda$--times $K_n(\Gamma')$. As $\Delta'$ is a
  $(\lambda d,e)$--regular bipartite graph, it admits
  by Lemma \ref{cor813}
  a full matching $M$ of size $\lambda|K_n(\Gamma')|$;
  in particular, each vertex in $\,^{\lambda}K_{n}(\Gamma')$
  is matched to exactly one vertex of $K_n(\Gamma)$.
\end{proof}
By collapsing the multiset $\,^{\lambda}K_n(\Gamma')$ in the proof of
the previous lemma, we have that $M$ associates
$\lambda$ elements of $K_n(\Gamma)$ to each element
of $K_n(\Gamma')$. This leads to the following
corollary.
\begin{corollary}
\label{lr1}
Suppose
\[ \lambda=\left\lfloor\frac{|K_n(\Gamma)|}{|K_n(\Gamma')|}\right\rfloor; \]
then,
there exists a sampling $\xi:K_n(\Gamma)\to K_n(\Gamma')$
such that any $g\in K_n(\Gamma')$ has \emph{at least}
$\lambda$ preimages.
\end{corollary}
It can be seen directly from Corollary \ref{lr1} that when
\[ |K_n(\Gamma)|=\lambda|K_n(\Gamma')|+1, \]
there always exists a $(\lambda,\lambda+1)$--semiregular sampling.
In general, however, further hypotheses on the nature of the
embedding of $K_n(\Gamma')$ into $K_n(\Gamma)$ are required. We
prove a result for the case $\lambda=1$.

\begin{theorem}
\label{Ref} Let $\Gamma'\leq\Gamma\leq K_n$ with
\begin{equation}
\label{r}
|K_n(\Gamma)|-|K_n(\Gamma')|=r<|K_n(\Gamma')|.
\end{equation}
Suppose
\begin{equation}
\label{kgp}
 |K_n(\Gamma)|>\frac{er^2}{e+r-1},
\end{equation}
where $e\in\NN$ is the number of elements of $K_n(\Gamma)$
containing $\Gamma'$. Then, there is a $(1,2)$--semiregular
sampling $\xi:K_n(\Gamma)\to K_n(\Gamma')$.
\end{theorem}
\begin{proof}
Argue as in the proof of Lemma \ref{tteo}, and
construct a $(d,e)$--regular bipartite graph
$\Delta$.
As
\begin{equation}
\label{deK}
 d|K_n(\Gamma)|=e|K_n(\Gamma')|,
\end{equation}
by \eqref{r} we have
\[ (e-d)|K_n(\Gamma')|=dr>0. \]
Thus, $e-1\geq d$. Also, by \eqref{r} and \eqref{deK}, we have
$er=(e-d)|K_n(\Gamma)|$, hence using \eqref{kgp} it results
\[ d>\frac{r-1}{r}(e-1). \]
Determine now, using Lemma \ref{cor813}, a full matching of
$\Delta$; this provides values for the function $\xi$ on a subset
$T$ of $K_n(\Gamma)$ with $T=|K_n(\Gamma')|$. Let now
$R=K_n(\Gamma)\setminus T$ be the set of the vertices of
$K_n(\Gamma)$ which have not been already matched, and consider
the bipartite graph $\Delta'=(V',E')$ obtained from $\Delta$ by
keeping just the vertices in $R\cup K_n(\Gamma')$. This graph, in
general, is not regular; however, $\deg_{\Delta'}(v)=d$ if $v\in
R$ and $\deg_{\Delta'}(v)<e$ if $v\in K_n(\Gamma')$. By Theorem
\ref{klct}, $\chi'(\Delta')\leq e-1$. Let $w:E'\to L\subset \NN$
be a colouring of $\Delta'$ with $|L|\leq e-1$. Since
$R=\Delta'\cap K_n(\Gamma)$ contains exactly $r$ vertices and each
of these is incident with at least $(e-1)(r-1)/r+1$ differently
coloured edges, there is at least an $\ell\in L$ such that each
vertex of $R$ is incident with an edge of colour $\ell$. Thus, the
set
\[ C_{\ell}=\{ x\in E': w(x)=\ell \} \]
is a full matching for $\Delta'$.
Now, for any $r\in R$,  define $\xi(r)=s$,
where $(r,s)\in C_{\ell}$.
This completes the proof.
\end{proof}

\begin{example}
  An interesting case for Theorem \ref{Ref} occurs
  when $d$ is taken to be as large as possible, namely
  $d=e-1$. Since, in this case, by \eqref{r} and \eqref{deK}
  \[ r=\frac{1}{e}|K_n(\Gamma)|, \]
 Condition \eqref{kgp}
  is always satisfied.
  We provide now an actual example where this
  happens.
  Suppose $n$ to be even and let $\Gamma=K_{n/2}$
  and $\Gamma'=K_{n/2-1}$.
  Each element of $K_n(\Gamma')$ is contained
  in $e=n/2+1$ elements of $K_n(\Gamma)$, while
  $\Gamma$ contains $d=n/2$ elements
  of $K_n(\Gamma')$.
\end{example}

\section{Examples and applications}
\label{s4} As said above, in this section we will show direct
constructions of samplings for some complete designs. A convenient
approach is to consider a suitable automorphism group of the
designs which is also compatible with the sampling we wish to
find, in the sense of the following definition.
\begin{definition}
  Let $\fB$ be a $(K_n,\Gamma)$--design and let $\fB'$ be a
  $(K_n,\Gamma')$-design, with $\Gamma'\leq \Gamma$,
  and suppose $\xi:\fB\to\fB'$ to be a sampling.
  An \emph{automorphism} $\alpha$ of $\xi$
  is an automorphism of $\fB'$ and $\fB$ such that for any
  $B\in\fB$,
  \[ \xi(\alpha(B))=\alpha(\xi(B)). \]
\end{definition}
Observe that an analogous definition is possible also for embeddings;
see, for instance \cite[Theorem 3.2]{B}
where a $2-(p,3,1)$ design is cyclically embedded into
a cyclic $2-(4p,4,1)$ design.

\noindent Let $n$ be an integer. From now on,
we shall mean by
$\ZZ_n^{\hbox{\tiny$\square$}}$  the group of all the
invertible elements in $\ZZ_n$ which are squares.

\begin{example}
\label{ex1}
  Fix $n$. Let $\Gamma=C_k\leq K_n$ be a cycle on $k$ vertices, and
  write $\Gamma'=P_h\leq C_k$ for a path in $C_k$ with $h$ vertices.
  We have
  \[ |K_n(C_k)|=\binom{n}{k}\frac{(k-1)!}{2};\qquad
  |K_n(P_h)|=\binom{n}{h}\frac{h!}{2}. \]
  By Theorem \ref{mainTG}, $K_n(C_k)$ admits a regular $P_h$--sampling,
  if, and only if,
  \[\lambda=\frac{|K_n(C_k)|}{|K_n(P_h)|}=
  \frac{(n-h)!}{(n-k)!k}\in\NN. \]
  In general,
  even with the existence of a group action compatible with the
  structures involved, it is not easy to actually determine a
  sampling.
  \par
  We write a full example just for $n=7$, $k=4$ and $h=3$; here,
  $\lambda=1$.
  Identify the vertices of $K_7$ with the elements of $\ZZ_7$.
  First, we wish to find a group $G$  which is
  \begin{enumerate}
  \item transitive on $K_7$;
  \item acts in a semiregular way on $K_7(C_4)$ and, possibly,
    on $K_7(P_3)$.
  \end{enumerate}

  Such a group $G$, if it exists, it is necessarily
  isomorphic to $H/\Stab_H(C)$ for some $H\leq S_7$ and
  any $C\in K_7(C_4)$.
  Thus, we need first to determine all
  $H\leq S_7$ normalising at least $\Stab_H(C_4)$.
  A direct computation shows that,
  up to conjugacy, there are just two classes of such
  subgroups:
  \begin{enumerate}[(a)]
  \item one consisting of
    cyclic groups of order $7$ isomorphic to $(\ZZ_7,+)$,
  \item the other containing groups of order $21$ isomorphic to
    $G=\ZZ_7^{\hbox{\tiny$\square$}}\ltimes\ZZ_7$,
    where, for
     $(\alpha,\beta),(\alpha',\beta')\in G$,
     \[ (\alpha,\beta)(\alpha',\beta') = (\alpha\alpha',\beta+\alpha\beta'). \]
     The action of $G$ on $V(K_7)=\ZZ_{7}$ is given by
    \[ (\alpha,\beta)(x)=\alpha x+\beta. \]
  \end{enumerate}
It might be checked that the
action induced by both these groups on $K_7(C_4)$ and
$K_7(P_3)$ is semiregular.

We now describe $2$ different samplings.

First consider the group $(\ZZ_7,+)$.
In this case, a complete system of representatives for $K_7(P_3)$
is given by  the paths
\[ \begin{array}{llllllllll}
  012 & 013 & 014 & 015 & 016 & 023 & 024 & 025 & 026 & 034 \\
  035 & 036 & 045 & 046 & 056. \\
\end{array} \]
Recall that two triples $abc$ and $def$ represent the same
path in $K_7(P_3)$ if, and only if, either $abc=def$ or $abc=fed$.
Suitable representatives for the orbits of $K_7(C_4)$ turn out to be
\[
    \newcommand{\x}[1]{\underbar{#1}}
 \begin{array}{lllll}
  \x{012}6 & \x{013}6 & \x{014}6 & \x{015}2 & \x{016}3 \\
  \x{023}1 & \x{024}5 & \x{025}3 & \x{026}1 & \x{034}6 \\
  \x{035}1 & \x{036}2 & \x{045}1 & \x{046}1 & \x{056}2
  \end{array}
  \]
  The subpath selected by the sampling $\xi$ is
  given, for each representative, by the underlined elements, which
  have to be
  taken in the order they actually appear.
  As above, observe that the elements of $K_7(C_4)$ are not
  sets; in particular, two quadruples $abcd$ and $a'b'c'd'$ represent
  the same graph if, and only if, they belong to the same orbit of
  $\ZZ_7^4$ under the
  action of $D_8$, the dihedral group of order $8$.
  \par
  If the group   $\ZZ_7^{\hbox{\tiny$\square$}}\ltimes\ZZ_7$ is chosen for
  the construction, then a complete system of representatives for
  $K_7(P_3)$ is just
  \[ \begin{array}{lllll}
  012 & 013 & 014 & 015 & 035.
  \end{array} \]
  A related system of representatives for $K_7(C_4)$ is determined as
\[  \newcommand{\x}[1]{\underbar{#1}}
\begin{array}{lllll}
  \x{012}6 & \x{013}6 & \x{014}6 & \x{015}4 & \x{035}1.
\end{array} \]
The image of $0126$, $0136$, $0146$ and $0351$ in the
sampling $\zeta$ induced by these representatives is the same as that
in the sampling $\xi$ described above; however, it is not possible
to choose $\zeta(0152)=015$ again, since $0152$ now
belongs to the orbit of $0146$; thus, its sample is uniquely
determined by $\zeta(0146)=014$ and it must be $\zeta(0152)=025$.
This shows that $G$ is an automorphism group of  $\zeta$,
but not of $\xi$.
\end{example}

If we take both the graphs $\Gamma$ and $\Gamma'$ to be complete,
we may apply Theorem \ref{mainTG} to the study of embeddings and
samplings of complete designs in the classical sense.

We remark that the problem of determining a regular
$K_1$--sampling with redundancy $1$ of $K_n(K_k)$ is exactly that
of determining a system of distinct representatives for the
$k$--subsets of a set of cardinality $n$; see \cite{MHall,PHall}.

For any finite set $S$ of cardinality $n$, denote
by $\binom{S}{k}$ the set of all subsets of $S$ with $k$
elements.
\begin{corollary}
\label{maintS}
There exists a regular $k'$--sampling
\[ \xi:\binom{S}{k}\to\binom{S}{k'} \]
if, and only if, there is $\lambda\in\NN$ such that
\[ \binom{n}{k}=\lambda\binom{n}{k'}. \]
The redundancy of this sampling is $\lambda$.
\end{corollary}

\begin{corollary}
  \label{teoH}
  Let $S$ be a finite set with $|S|=n$. Suppose $k\leq\lfloor n/2\rfloor$.
  Then, there exists a bijective sampling
  \[ \xi:\binom{S}{n-k}\to \binom{S}{k}. \]
\end{corollary}

\begin{remark}
Corollary \ref{maintS} guarantees the existence of a regular $k'$--sampling of
$\binom{S}{k}$
with redundancy $\lambda$ every time the necessary condition
\[ \binom{n}{k}=\lambda\binom{n}{k'} \]
holds; yet our proof is  non--constructive.
\par
In at least some cases, however,
it might possible to write at least
some sampling functions $\xi$ in a more direct way.

The main idea, as before, is to describe both
$\binom{S}{k}$ and $\binom{S}{k'}$
as union of orbits under the action of a suitable  group
$G$, acting on $S$, and determine systems of  representatives
$T$ and $U$
such that:
\begin{enumerate}
\item $T\subseteq \binom{S}{k}$, $U\subseteq\binom{S}{k'}$;
\item $G$ is semiregular on $\binom{S}{k}$;
\item any element of $u\in U$ is a sample of
  ${\lambda}/{|\Stab_G(u)|}$
  elements of $T$;
\item for any $\sigma\in G$ and $t\in T$,
  \[ \sigma(\xi(t))=\xi(\sigma(t)). \]
\end{enumerate}
\end{remark}
\begin{example}
\label{ex1s}
  Suppose $k=3$. We are looking for a regular $2$--sampling
  of $\binom{S}{3}$; thus, $\lambda=(n-2)/3$.
  By Corollary \ref{maintS} such a regular $2$--sampling $\xi$ exists if,
  and only if, $n\equiv 2\pmod 3$.
  Observe that when $n$ is even, $\lambda$ is also even.

  In order to explicitly find $\xi$,
  consider the natural action of
  the cyclic group $(\ZZ_n,+)$ on the set $S=\{0,1,\ldots,n-1\}$:
  for any $\eta\in\ZZ_n$ and $s\in S$, let
  \[ \eta(s)=s+\eta \pmod n. \]
  Fix $v=\lfloor n/2\rfloor$.
  It is easy to show that
  a complete system of representatives for the orbits of
  $\ZZ_n$ on $\binom{S}{3}$ is given by
  either $T=T_1\cup T_2\cup T_3\cup T_4^1$ for $n$ odd,
  or
  $T=T_1\cup T_2\cup T_3\cup T_4^2$ for $n$ even,
  where
  \begin{equation}
    \label{eq3}
  \begin{array}{lll}
    T_1&=&\{ \{0,i,i+t\}\,|\, i=1,\ldots,\lambda+1; t=1,\ldots,\lambda \} \\
    T_2&=&\{ \{0,j,u\}\,|\, j=\lambda+2,\ldots,v-1; u=1,\ldots,j-\lambda-1 \} \\
    T_3&=&\{ \{0,\ell,m\}\,|\, \ell=\lambda+2,\ldots, v-1;
    m=\ell+1,\ldots, 2\lambda+1  \} \\
    T_4^1&=&\{ \{0,v,u\}\,|\, u=1,\ldots,v-\lambda-1,v+1,\ldots,2\lambda+1 \} \\
    T_4^2&=&\{ \{0,v,p\}\,|\, p=1,\ldots,\lambda/2 \}. \\
    \end{array}
  \end{equation}
  A set of representatives for the orbits of $\ZZ_{n}$ on $\binom{S}{2}$
  is just  $U=\{ \{0,x\} : 1\leq x\leq v \}$.
  All orbits of $\ZZ_n$ on $\binom{S}{3}$ have length $n$;
  thus $\ZZ_n$ acts semiregularly on $\binom{S}{3}$.
  When $n$  is odd, it is also true that
  $\Stab_{\ZZ_n}(y)=\{0\}$ for any $y\in U$.
  However, when $n$ is even,
  $\Stab_{\ZZ_n}(y)=\{0\}$ for $y\neq\{0,v\}$
  but
  $\Stab_{\ZZ_n}(y)=\{0,v\}$ when $y=\{0,v\}$.
  We now introduce a function $\widehat{\xi}:T\to U$
  such that  each $y\in U$, $y\neq\{0,v\}$ has
  $\lambda$ preimages in $T$, while $\{0,v\}$ for $v$ even
  has $\lambda/2$ preimages.
  Indeed,
  for each element $\{0,\ell,m\}$ in $T_i$, where $\ell$ and $m$
  are to be  taken in the same order as they appear in \eqref{eq3}, let
  \[ \widehat{\xi}(\{0,\ell,m\})=\{0,\ell\}. \]
  The group $\ZZ_n$ is semiregular on $\binom{S}{3}$ and $T$ is a set
  of representatives for its orbits. Hence, for any
  $\{a,b,c\}\in \binom{S}{3}$ there are
  unique $\{0,\ell,m\}\in T$ and  $\eta\in\ZZ_n$
   such that
  $\eta(\{0,\ell,m\})=\{a,b,c\}$.
  Thus, the definition
    \[ \xi(\{a,b,c\})=
    \xi(\eta(\{0,\ell,m\}))=\eta(\widehat{\xi}(\{0,\ell,m\})=
    \eta(\{0,\ell\})=\{\eta,\ell+\eta\} \]
    is well posed.
    We claim that $\xi$ is a regular $2$--sampling of $\binom{S}{3}$.
    Since for any
    $\{s,t\}\in \binom{S}{2}$ there exists $\eta\in\ZZ_n$ such that
    $\{s,t\}=\eta(\{0,\ell\})$,  in order to show that
    $\xi$ is a sampling we just need to prove that the
    number of preimages in $\binom{S}{3}$ of $\{0,\ell\}$ under
    $\widehat{\xi}$ is exactly $\lambda$.
    Observe that for $n$ odd or $\ell\neq v$,
    the only preimages of $\{0,\ell\}$
    are elements of $T$; thus, we have to analyse the following
    cases:
    \begin{enumerate}
    \item for $1<\ell\leq\lambda+1$, the set
      $\{0,\ell\}$ has $\lambda$ preimages in the class $T_1$
      and none in $T_2$, $T_3$, $T_4^1$ or $T_4^2$ and we are done;
    \item for $\lambda+2\leq\ell< v$, the set $\{0,\ell\}$ has
      $\ell-\lambda-1$ preimages in $T_2$ and
      $2\lambda+1-\ell$ preimages in $T_3$, for a total of
      $\lambda$;
    \item if $\ell=v$ and $n$ is odd, then $\{0,\ell\}$
      has $\lambda$ preimages in $T_4^1$.
    \end{enumerate}
    In $n$ is even and $\ell=v$, then each orbit of an
    element of $T_4^2$ contains two preimages of $\{0,v\}$;
    since $|T_4^2|=\lambda/2$, we get that also in this
    case the total number of preimages is $\lambda$.
    It follows that $\xi$ is, as requested, a regular
    $2$--sampling for $\binom{S}{3}$.

    We now show how this construction might be used for some small cases:
    \begin{enumerate}
    \item for $n=14$ we have $\lambda=4$, $v=7$. The set $T$ is as follows:
        \[
        \begin{array}{l@{:}l}
          T_1 &
            \begin{array}[t]{llllllllll}
          \underline{01}2 & \underline{01}3 & \underline{01}4 & \underline{01}5
          & \underline{02}3 &
          \underline{02}4 & \underline{02}5 & \underline{02}6 &
          \underline{03}4 & \underline{03}5 \\
          \underline{03}6 & \underline{03}7 &
          \underline{04}5 & \underline{04}6 &
          \underline{04}7 &
          \underline{04}8 &
          \underline{05}6 & \underline{05}7 &
          \underline{05}8 & \underline{05}9
          \end{array}
            \\
            T_2 &
            \begin{array}{lll}
              \underline{06}7 &  \underline{06}8 & \underline{06}9
            \end{array} \\
            T_3 &
            \begin{array}{l} \underline{06}1 \end{array} \\
            T_4^2 &
            \begin{array}{ll}
              \underline{07}1 &
              \underline{07}2 \end{array}
            \end{array}
            \]
            The underlined elements in the preceding table are
            the image under $\widehat{\xi}$ of the corresponding set.
          \item for $n=17$ we have $\lambda=5$, $v=8$.
            We describe $T$ and $\widehat{\xi}$ as before.
            In the following table, by $a$ and $b$ we respectively mean
            $10$ and $11$.
        \[
        \begin{array}{l@{:}l}
        T_1 &   \begin{array}[t]{llllllllll}
          \underline{01}2 & \underline{01}3 & \underline{01}4 &
          \underline{01}5 & \underline{01}6 & \underline{02}3 &
          \underline{02}4 & \underline{02}5 & \underline{02}6 &
          \underline{02}7 \\
          \underline{03}4 & \underline{03}5 & \underline{03}6 &
          \underline{03}7 & \underline{03}8 & \underline{04}5 &
          \underline{04}6 & \underline{04}7 & \underline{04}8 &
          \underline{04}9 \\
          \underline{05}6 & \underline{05}7 & \underline{05}8 &
          \underline{05}9 & \underline{05}a & \underline{06}7 &
          \underline{06}8 & \underline{06}9 & \underline{06}a &
          \underline{06}b
        \end{array} \\
        T_2 & \begin{array}{lllllll}
          \underline{07}8 & \underline{07}9 & \underline{07}a &
          \underline{07}b \\
        \end{array} \\
        T_3 & \begin{array}[t]{l}
          \underline{07}1 \\
        \end{array} \\
        T_4^1 & \begin{array}{lllll}
          \underline{08}1 & \underline{08}2 &
          \underline{08}9 & \underline{08}a & \underline{08}b
        \end{array}
      \end{array}
      \]
    \end{enumerate}
  \end{example}
  \begin{example}
    \label{ex2}
    We wish to determine a regular $3$--sampling $\xi$ of $\binom{S}{4}$.
    In this case,
    $k=4, \lambda=5$. To construct $\xi$ it is convenient to describe
    $\binom{S}{4}$ as union of orbits of a fairly large semiregular group.
    We may consider the group
    \[ G=\ZZ_n^{\hbox{\tiny$\square$}}\ltimes\ZZ_n. \]
    For $n$ a prime with $n\equiv 11\pmod{12}$,
    simple, but tedious, computations show that the
    action of $G$ is semiregular on both $\binom{S}{3}$ and $\binom{S}{4}$.
    The smallest interesting case occurs for $n=23$.
    Here, a system of representatives for the orbits of $\binom{S}{3}$ is
    \[
    U=\{
    \begin{array}{lllllll}
      0,1,2 & 0,1,3 & 0,1,4 & 0,1,5 & 0,1,7 & 0,1,9 & 0,1,13
    \end{array}\}.
    \]
    By a computer assisted search with \cite{GAP4},
    we determined the following
    system compatible with $U$ of representatives for
    the orbits of $\binom{S}{4}$:
    \[
    \newcommand{\x}[1]{\underbar{#1}}
    \begin{array}{llllll}
      \x{0,1,2},5 & \x{0,1,2},7 & \x{0,1,2},10 & \x{0,1,2},11 & \x{0,1,2},14 \\
      \x{0,1,3},7 & \x{0,1,3},15 & \x{0,1,3},19 & \x{0,1,3},21 & \x{0,1,3},22 \\
      \x{0,1,4},5 & \x{0,1,4},7  & \x{0,1,4},11 & \x{0,1,4},15 & \x{0,1,4},17 \\
      \x{0,1,5},6 & \x{0,1,5},14 & \x{0,1,5},15 & \x{0,1,5},20 & \x{0,1,5},22 \\
      \x{0,1,7},5 & \x{0,1,7},9  & \x{0,1,7},10 & \x{0,1,7},21 & \x{0,1,7},22 \\
      \x{0,1,9},2 & \x{0,1,9},5  & \x{0,1,9},8  & \x{0,1,9},16 & \x{0,1,9},20 \\
      \x{0,1,13},3& \x{0,1,13},5 & \x{0,1,13},7 & \x{0,1,13},9 & \x{0,1,13},12.
    \end{array}
    \]
    The sampling map $\widehat{\xi}$ is defined as in the previous example.
    \par
    We remark that, instead of the
    group $G$, we might also have  considered
    the action of the cyclic group $\ZZ_{23}$. However, had this
    been the case, we would have needed to write
    $77$ distinct representatives for the
    orbits of $\binom{S}{3}$ and $385$ compatible representatives for the
    orbits of $\binom{S}{4}$.
  \end{example}
  \begin{remark}
    If $\binom{S}{k}$ admits a regular $k_1$--sampling $\xi_1$
    with redundancy $\lambda_1$ and $\binom{S}{k_1}$ admits, in turn, a regular
    $k_2$--sampling $\xi_2$ with redundancy $\lambda_2$, then
    $\xi=\xi_2\xi_1$ is a regular $k_2$--sampling of $\binom{S}{k}$
    redundancy $\lambda_1\lambda_2$,
    since any $x_2\in \binom{S}{k_2}$ is a sample of $\lambda_2$ elements
    of $\binom{S}{k_1}$, while, on the other hand, any $x_1\in \binom{S}{k_1}$ is a sample
    of $\lambda_1$ elements of $\binom{S}{k}$.
    However, it has to be stressed
    that not all $k_2$--samplings of $\binom{S}{k}$
    arise in this way.
  \end{remark}
  \begin{example}
    For $n=11$,  the necessary condition for the existence
    of a regular $2$--sampling of $\binom{S}{3}$ as well as that
    for the existence
    of a regular $3$--sampling of $\binom{S}{4}$ are simultaneously fulfilled.
    In particular, it is possible
    to construct a regular $2$--sampling $\xi$ to $\binom{S}{4}$ with
    redundancy $\lambda=6$
    as composition of a regular $2$--sampling $\xi_1$ of $\binom{S}{3}$ and a
    regular $3$--sampling $\xi_2$ of $\binom{S}{4}$.
    We consider the action of the group $G$ introduced in Example
    \ref{ex2}. As before, $G$ is semiregular on $\binom{S}{4}$ and
    $\binom{S}{3}$.
    Furthermore, a direct computation
    shows that $G$ is regular on $\binom{S}{2}$.
    We provide suitable systems of representatives for, respectively, $\binom{S}{2}$,
    $\binom{S}{3}$ and $\binom{S}{4}$:
    \[
    \newcommand{\x}[1]{\underbar{#1}}
    \begin{array}{l@{:}l}
      U_2&  \begin{array}{l} 01 \end{array}   \\
      U_3&  \begin{array}{lll} \x{01}2 & \x{01}3 & \x{01}5 \end{array}  \\
      U_4& \begin{array}{llllll}
        \x{012}3 & \x{012}4 & \x{013}5 & \x{013}7 & \x{015}4 &
        \x{015}8.
      \end{array}
    \end{array} \]
    The samplings arise, as in the previous examples,
    from $\widehat{\xi_1}:U_3\to U_2$ and $\widehat{\xi_2}:U_4\to U_3$.
    It is immediate to see that $\widehat{\xi}=\widehat{\xi_1}\widehat{\xi_2}$
    is a regular $2$--sampling of $\binom{S}{4}$.
    On the other hand, it is possible to define a regular $2$--sampling of
    $\binom{S}{4}$ also from the starter set
    \[
    \newcommand{\x}[1]{\underbar{#1}}
    \begin{array}{l@{:}l}
    U_4& \begin{array}{llllll} \x{01}23 & \x{01}24 & \x{01}25
      & \x{01}26 & \x{01}28 & \x{01}34
      \end{array}
    \end{array}. \]
      However, it is not possible to extract  a regular
      $3$--sampling (with redundancy $2$) from $U_4$.
  \end{example}

The notion of $(K_n,\Gamma)$--complete design can be further
generalized to that of a $(\Delta,\Gamma)$--complete design, where
$\Delta$ is an arbitrary graph. As before, given
$\Gamma'\leq\Gamma\leq\Delta$, we might want to investigate the
existence of a regular $\Gamma'$--sampling of $\Delta(\Gamma)$ or,
conversely, an embedding of $\Delta(\Gamma')$ into
$\Delta(\Gamma)$. However, in this general case,  Lemma \ref{tteo}
fails, since  it is not possible to guarantee that $\Aut(\Delta)$
acts transitively on the blocks of $\Delta(\Gamma)$ and
$\Delta(\Gamma')$; hence it is not possible to get an analogue of
Theorem \ref{mainTG}. The following example contains a case in
which a sampling might be shown to exist (and at least one of
these samplings can be determined in an independent way).

\begin{example}
\label{PG}
  Let $q$ be any prime power, and consider the projective
  space $PG(3,q)$. Call $\hat{\Delta}$ the bipartite
  point--line incidence
  graph of this geometry.
  Let $\widehat{\Gamma}$ be the point--line incidence graph
  of $PG(2,q)$, seen as a plane embedded into $PG(3,q)$.
  Clearly $\widehat{\Gamma}<\widehat{\Delta}$.
  Define two new graphs $\Delta$ and $\Gamma$ by
  replacing in $\widehat{\Delta}$ and $\widehat{\Gamma}$
  each vertex $v$ corresponding to a point of $PG(3,q)$
  with a triangle $vv'v''$.
  Let $\Gamma'$ be a triangle $C_3$. We observe
  that $\Delta(\Gamma')$ corresponds to the set of the
  points of $PG(3,q)$. Since there are as many points
  in $PG(3,q)$ as planes, we have
  $|\Delta(\Gamma')|=|\Delta(\Gamma)|$.
  Furthermore, the full automorphism group of $\Delta$
  contains a subgroup isomorphic to $PGL(3,q)$, which
  acts in a transitive way on $\Delta(\Gamma)$ and
  also $\Delta(\Gamma')$. Thus, we may apply the
  same techniques as in the proof of Lemma \ref{tteo},
  as to obtain a regular matching of $\Delta$,
  as in Theorem \ref{mainTG}. This gives a sampling
  that associates to
  every plane $\pi$ of $PG(3,q)$ a point $p\in\pi$.
  \par
  There are several possible matchings
  of this kind; one of these is
  given
  by the action of a symplectic polarity $\sigma$
  of $PG(3,q)$.
\end{example}


\begin{thebibliography}{999}

\bibitem{A} Aigner, M.,
  ``{Combinatorial theory}'',
Classics in Mathematics, \emph{ Springer-Verlag}  (1997).


\bibitem{BO} Bos\'ak, J.,
  ``Decompositions of graphs'',
  Mathematics and its Applications (1990),
  \emph{Kluwer Academic Publishers Group}.

\bibitem{BZ} Bryant, D., El-Zanati, S.,
  \emph{Graph Decompositions}, CRC Handbook of Combinatorial
  Designs, C.J. Colbourn and J.H. Dinitz, CRC Press (2006),
  477--486.

\bibitem{BR} Bryant, D., Rodger, C.,
  \emph{Cycle Decompositions}, CRC Handbook of Combinatorial
  Designs, C.J. Colbourn and J.H. Dinitz, CRC Press (2006),
  373--382.


\bibitem{B} Buratti, M., \emph{Cyclic designs with block size 4 and related
optimal optical orthogonal codes}, Des. Codes Cryptogr.
\textbf{26} (2002), 111--125.

\bibitem{Cam} Cameron, P.J.,
  ``Parallelisms of Complete Designs'', Cambridge University
  Press (1976).

\bibitem{CLQ}  Colbourn, C. J., Ling, A. C. H., Quattrocchi, G.,
\emph{Minimum embedding of $P\sb 3$-designs into $(K\sb
4-e)$-designs}, J. Combin. Des. \textbf{11} (2003), 352--366.

\bibitem{GAP4}
  The GAP~Group, \emph{GAP -- Groups, Algorithms, and Programming,
    Version 4.4.12};
  2008,
  \verb+(http://www.gap-system.org)+.


\bibitem{GL} Gionfriddo, L., Lindner, C. C., \emph{Nesting kite and 4-cycle
systems}, Australas. J. Combin. \textbf{33} (2005), 247--254

\bibitem{GQ}  Gionfriddo, M., Quattrocchi, G., \emph{Embedding balanced
$P\sb 3$-designs into (balanced) $P\sb 4$-designs},
  Discrete Math. \textbf{308} (2008), 155--160.

\bibitem{GMR} Granville, A., Moisiadis, A., Rees, R.,
\emph{Nested Steiner $n$-cycle systems and perpendicular arrays},
J. Combin. Math. Combin. Comput. \textbf{3} (1988), 163--167.

\bibitem{MHall} Hall, M. Jr.,
``{Combinatorial theory}'',
Wiley Classics Library.  \emph{John Wiley \& Sons, Inc.} (1998).

\bibitem{PHall} Hall, P.,
\emph{On representatives of subsets}, J. London Math. Soc. {\bfseries 10}
(1935), pp. 26--30.

\bibitem{H} Harary, F.,
  ``Graph Theory'', Addison--Wesley Publishing Co. (1969).

\bibitem{K1} K\H{o}nig, D., \emph{\"{U}ber {G}raphen und ihre {A}nwendung auf
              {D}eterminantentheorie und {M}engenlehre},
            Math. Ann. {\bfseries 77} (1916), 453--465.

\bibitem{K2} K\H{o}nig, D., \emph{Graphok \'es alkalmaz\'asuk a determin\'ansok
  \'es a halmazok elm\'elet\'ere}, Math. Term\'esz. \'Ert. {\bfseries 34}
  (1916), 104--119.


\bibitem{LQR}  Lindner, C. C., Quattrocchi, G., Rodger, C. A., \emph{Embedding
Steiner triple systems in hexagon triple systems},
  Discrete Math. \textbf{309} (2009), 487--490.

\bibitem{LRS89} Lindner, C. C., Rodger, C. A., Stinson, D. R.,
\emph{Nesting of cycle systems of odd length}, Combinatorial
designs---a tribute to Haim Hanani. Discrete Math. \textbf{77}
(1989), 191--203.

\bibitem{LRS}  Lindner, C. C., Rodger, C. A., Stinson, D. R.,
\emph{Nestings of directed cycle systems}, Ars Combin. \textbf{32}
(1991), 153--159.



\bibitem{LS} Lindner, C. C., Stinson, D. R., \emph{Nesting of cycle systems
of even length}, J. Combin. Math. Combin. Comput. \textbf{8}
(1990), 147--157.

\bibitem{LCQ} Ling, A. C. H., Colbourn, C. J., Quattrocchi, G.,
\emph{Minimum embeddings of Steiner triple systems into $(K\sb
4-e)$-designs. II},
  Discrete Math. \textbf{309} (2009), 400--411.


\bibitem{LP} Lov\'asz L., Plummer, M.D., ``Matching Theory'',
  Annals of Discrete Mathematics {\bfseries 29} (1986).


\bibitem{MR} Mathon, R., Rosa, A.,
  \emph{$2-(v,k,\lambda)$ Designs of Small Order},
  CRC Handbook of Combinatorial
  Designs, C.J. Colbourn and J.H. Dinitz, CRC Press (2006),
  25--58.



\bibitem{MQ95} Milici, S., Quattrocchi, G., Shen, H.,
\emph{Embeddings of simple maximum packings of triples with
$\lambda$ even}, Discrete Math. \textbf{145} (1995), 191--200.


\bibitem{MQ99} Milici, S., Quattrocchi, G., \emph{Embedding handcuffed
designs with block size $2$ or $3$ in $4$-cycle systems},
   Combinatorics (Assisi, 1996). Discrete Math. \textbf{208/209} (1999), 443--449

\bibitem{M}  Milici, S., \emph{Minimum embedding of $P\sb 3$-designs into
${\rm TS}(v,\lambda)$},
  Discrete Math. \textbf{308} (2008), 331--338.


\bibitem{AP} Pasotti, A.,
  ``Graph Decompositions with a sharply vertex transitive automorphism
  group'', Ph.D. Thesis, Universit\`a degli Studi di Milano Bicocca
  (2006).

\bibitem{Q2002a} Quattrocchi, G., \emph{Embedding $G_1$--designs into $G_2$--designs, a short survey},
Rend. Sem. Mat. Messina Ser. II \textbf{8} (2001), 129--143.

\bibitem{Q2002} Quattrocchi, G., \emph{Embedding path designs in 4-cycle
systems},
Combinatorics '98 (Palermo). Discrete Math. \textbf{255} (2002),
349--356.

\bibitem{Q2003} Quattrocchi, G., \emph{Embedding handcuffed designs in $D$-designs,
where $D$ is the triangle with attached edge}, Discrete Math.
\textbf{261} (2003), 413--434.

\bibitem{RS} Rodger, C. A., Stinson, D. R., \emph{Nesting directed cycle systems
of even length},
 European J. Combin. \textbf{13} (1992), 213--218.

\bibitem{S} Stinson, D. R., \emph{The construction of nested cycle
systems},
Coding theory and design theory, Part II, 362--367,
 IMA Vol. Math. Appl., \textbf{21}, Springer, New York, 1990.

\bibitem{T} Tutte W.T., ``Graph Theory'', Cambridge University Press
  (2001).


\end{thebibliography}
\end{document}